\documentclass[12p]{article}

\usepackage[T1]{fontenc}

\usepackage{amsmath,amsfonts,amscd,amssymb}
\usepackage{graphicx}
\usepackage{enumerate}

\usepackage{cleveref}
\usepackage{filecontents}

\newtheorem{theorem}{Theorem}[section]
\newtheorem{definition}[theorem]{Definition}
\newtheorem{corollary}[theorem]{Corollary}
\newtheorem{proposition}[theorem]{Proposition}
\newtheorem{lemma}[theorem]{Lemma}

\newtheorem{remark}[theorem]{Remark}

\newtheorem{example}[theorem]{Example}
\newcommand{\pr}{\indent{\em Proof: \ }}
\newcommand{\qed}{\hspace*{5 mm}$\square$\bigskip}
\newenvironment{proof}{\noindent {\pr}\ }{\qed}
\newcommand{\Z}{{\mathbb{Z}}}

\newcommand{\F}{\mathbb{F}}

\newcommand{\D}{{\cal D}}

\newcommand{\cS}{{\cal S}}

\newcommand{\bv}{\mathbf{v}}
\newcommand{\ba}{\mathbf{a}}
\newcommand{\bb}{\mathbf{b}}

\newcommand{\bA}{\mathbf{A}}
\newcommand{\bB}{\mathbf{B}}

\newcommand{\bu}{\mathbf{u}}
\newcommand{\bw}{\mathbf{w}}
\newcommand{\be}{\mathbf{e}}

\newcommand{\bKappa}{\mathbf{K}}

\newcommand{\bkappa}{\boldsymbol\kappa}

\newcommand{\wt}{\operatorname{wt}}
\newcommand{\Aut}{\operatorname{Aut}}

\newcommand{\Supp}{\operatorname{Supp}}
\newcommand{\rank}{\operatorname{rank}}
\newcommand{\kernel}{\operatorname{ker}}
\newcommand{\Gcd}{\operatorname{gcd}}

\newcommand{\HFP}{\operatorname{HFP}}
\newcommand{\op}{\operatorname{op}}
\title{Hadamard Full Propelinear Codes of type Q. Rank and Kernel\thanks{Submitted to \textit{Designs, Codes and Cryptography}}}

\author{J. Rif\`{a} \and
E. Su\'arez Canedo\thanks{J. Rif\`{a}, E. Su\'arez Canedo at Department of Information
	and Communications Engineering, Universitat Aut\`{o}noma de Barcelona,
	08193-Bellaterra, Spain 
}}

\begin{document}
	
\maketitle

\begin{abstract}
Hadamard full propelinear codes ($\HFP$-codes) are introduced and their equivalence with Hadamard groups is proven (on the other hand, it is already known the equivalence of Hadamard groups with relative $(4n,2,4n,2n)$-difference sets in a group and also with cocyclic Hadamard matrices).
We compute the available values for the rank and dimension of the kernel of $\HFP$-codes of type Q and we show that the dimension of the kernel is always 1 or $2$.
We also show that when the dimension of the kernel is 2 then the dimension of the kernel of the transposed code is 1 (so, both codes are not equivalent). Finally, we give a construction method such that from an $\HFP$-code of length $4n$, dimension of the kernel $k=2$, and maximum rank $r=2n$, we obtain an $\HFP$-code of double length $8n$, dimension of the kernel $k=2$, and maximum rank $r=4n$.


\end{abstract}

\section{Introduction}

Let $\F$ be the binary field. For any $v\in \F^n$, we define the support of $v$ as the set of nonzero positions of $v$ and we denote it by $\Supp(v)$.
Denote by $\wt(x)$ the \textit{Hamming weight} of a vector $x\in \F^n$ (i.e. the number
of its nonzero positions).
Given two vectors $x=(x_1,\ldots, x_n)$ and $y = (y_1, \ldots, y_n)$
from $\F^n$ we denote by $d(x,y)$ the \textit{Hamming distance} between $x$
and $y$ (i.e., the number of positions $i$, where $x_i \neq y_i$).
Let us denote by $e_i\in \F^n$ the vector with the value of all coordinates zero, except the $i$th which is one. A binary {\it code} $C$  of length $n$ is a nonempty subset of $\F^n$. The
elements of $C$ are called {\it codewords}.
Two structural properties of (nonlinear) codes are the dimension of the linear
span and kernel. The \textit{rank} of a binary code $C$, $r=\rank(C)$, is the dimension of the linear span of $C$. The \textit{kernel} of a binary code $C$ is the set of codewords which keeps the code invariant by translation, $K(C) := \{z \in \F^n : C + z = C\}$.
Assuming the zero vector is in $C$ we have that $K(C)$ is a linear subspace.
We denote the dimension of the kernel of $C$ by $k=\kernel(C)$.
For a binary code $C$, $\Aut(C)$ denotes the group of automorphisms of $C$, i.e.
the set of coordinate permutations that fixes $C$ setwise.

An \emph{Hadamard matrix} of order $n$ is a matrix of size $n \times n$ with entries $\pm 1$, such that $HH^T=n I$, where $I$ is the identity matrix and $H^T$ means the transpose matrix of $H$. When $n>1$ we can easily see that any two different rows (columns) of an Hadamard matrix agree in precisely $n/2$ coordinates and when $n>2$ any three different rows (columns) agree in precisely $n/4$ coordinates. Thus, for $n>2$, the order of an Hadamard matrix is always a multiple of 4. It is conjectured that the converse holds, i.e., there are Hadamard matrices of order $4n$, for any positive integer $n$ \cite{ak}.

Two \emph{Hadamard matrices} are \emph{equivalent} if one can be obtained from the other by permuting rows (or columns) or multiplying rows (or columns) by $-1$. With these operations we can change the first
row and column of $H$ into $+1$'s and we obtain an equivalent Hadamard matrix which is called \textit{normalized}. If $+1$'s are replaced by $0$'s and $-1$'s by $1$'s, the initial Hadamard matrix is changed into a (binary) Hadamard matrix and, from now on, we will refer to it when we deal with Hadamard matrices. The binary code consisting of the rows of a (binary) Hadamard matrix and their complements is called a (binary) \emph{Hadamard code}, which has $8n$ codewords, length $4n$ and minimum distance $2n$.

Elliott and Butson \cite{eb} define a \textit{relative $(v, m, k,\lambda)$-difference set in a group} $G$ relative
to a normal subgroup $N\lhd G$, where $|G| = vm$ and $|N| = m$, as a subset $D$ of $G$ with $|D|=k$ such that the multiset of values $d_1d_2^{-1}$
of distinct elements $d_1,d_2\in D$
contains each element of $G\backslash N$ exactly $\lambda$ times, and contains no elements of N. Thus $k(k-1) = \lambda m(v-1)$ and $v\not= 2k$. Equivalently,
$|D\cap xD| = \lambda$, for all $x\in G\backslash N$.

Let $D$ be a relative $(4n,2,4n,2n)$-difference set in a group $G$ of order $8n$ relative to a normal subgroup $N\simeq \Z_2$ of $G$. Such a group is called an \textit{Hadamard group} of order $8n$, in fact, a \textit{left Hadamard group} using the following fine-tuned definition from~\cite{ito}.

\begin{definition}\label{hg}
	A triple $(G,D,u)$ is a left Hadamard group of order $8n$ if $G$ is a finite group containing a prescribed $4n$-subset $D$ and a prescribed central involution $u$ ($D$ is called the Hadamard subset corresponding to $u$), such that
	\begin{enumerate}[(i)]
		\item $a D$ and $D$ intersect in exactly $2n$ elements, for any $a\notin \langle u\rangle\subset G$,
		\item $a D$ and $\{b,bu\}$ intersect in exactly one element, for any $a,b\in G$.
	\end{enumerate}
\end{definition}
Note that taking $a$ in (ii) as the identity element of $G$ we obtain that  $D$ and $u D$ are disjoints and $D\cup u D = G$.

A right Hadamard group $(G,D,u)$ can be characterized as a left Hadamard group over the opposite group $G^{\op}$ of $G$ (the opposite group $G^{\op}$ of $G$, has the same underlying set as $G$ and its group operation $\diamond$ is defined by $a\diamond b=b*a$).

From now on, when we use the term Hadamard group without any specification, we are referring to a left Hadamard group.

Hadamard matrices corresponding to Hadamard groups can also be obtained from $2$-cocycles~\cite{lfh,flan}.
The concept of cocyclic Hadamard matrix was introduced in~\cite{hl} and in~\cite{flan} it is proven that cocyclic Hadamard matrices are equivalent to Ito's Hadamard groups. In~\cite{itoIII}, a special Hadamard group was introduced, called type Q. In this case $G$ is a group of order $8n$, $G=\langle \ba, \bb : \ba^{4n}=e, \ba^{2n}=\bb^2, \bb^{-1}\ba\bb=\ba^{-1}\rangle$ with only one involution $u=\ba^{2n}=\bb^2$ which is central in $G$ and where $\be$ means the neutral element.

Hadamard conjecture asserts that an Hadamard matrix of order $4n$ exists for every positive integer $n$. The smallest order for which no Hadamard matrix is known is 668, and at the time of~\cite{hor} the smallest order for which no cocyclic Hadamard matrix is known is 188. Also, in~\cite{itoIII}, Ito conjectured that relative $(4n,2,4n,2n)$-difference sets in  groups of type Q exists for all positive integers $n$ and he shows it is true for $n\leq 11$. Later~\cite{sch} Ito's conjecture was verified for $n\leq 46$.

Let ${\cal S}_n$ denote the symmetric group of permutations of the set $\{1,\ldots, n\}$.  We write $\pi(v)$ to denote $(v_{\pi^{-1}(1)}, v_{\pi^{-1}(2)}, \ldots, v_{\pi^{-1}(n)})$, where $\pi \in {\cal S}_n$ and $v=(v_1,v_2,\ldots,v_n) \in \F^n$.

When we talk about a Hadamard group $G$ we use $e$, $u$ to refer the neutral element and the central involution in $G$, respectively. When we talk about binary codes we denote $\be, \bu$ the all zeros and the all ones vector, respectively. 

\begin{definition}[\cite{rbh}]\label{def:propelinear}
	A binary code $C$ of length $n$, such that $\be \in C$, has a  \textbf{propelinear} structure if for each codeword
	$x\in C$ there exists $\pi_x\in \cS_n$ satisfying the following conditions:
	\begin{enumerate}
		\item For all $x,y\in C$, $x+\pi_x(y)\in C$,
		\item For all $x,y\in C$, $\pi_x\pi_y=\pi_z$, where $z = x +\pi_x(y)$.
	\end{enumerate}
\end{definition}

For all $x\in C$ and for all $y\in \Z_2^n$, denote by $\cdot$ the binary operation
such that $x\cdot y=x+\pi_x(y)$. Then, $(C,\cdot)$ is a group, which is not
abelian in general. The zero vector $\be$ is the identity element and $\pi_{\be}=I$ is the
identity permutation. Moreover, $x^{-1}=\pi_x^{-1}(x)$, for all $x\in C$~\cite{rbh}. We call $C$ a propelinear code if it has a propelinear structure.

The current paper is focused on the study and computation of the available parameters for the values of the rank and the dimension of the kernel of Hadamard codes corresponding to Hadamard groups of type Q.
In \Cref{sec:2} we introduce the $\HFP$-codes (Hadamard full propelinear codes) and we show they are equivalent to Hadamard groups. We also show some properties for the kernel of these codes.
In \Cref{sec:3} we introduce the concept of $\HFP$-codes of type Q, which correspond to Hadamard groups of type Q. We study the available values for the rank and the dimension of the kernel of $\HFP$-codes of type Q and we show that the dimension of the kernel is always 1 or $2$.
In \Cref{sec:4} we characterize $\HFP$-codes with dimension of the kernel $k=2$ and we show that the transposed matrix of an Hadamard matrix of type Q with $k=2$ has the dimension of the kernel equal to 1, so both Hadamard matrices are not equivalent.
Finally, in \Cref{sec:5}, we give a construction such that from an $\HFP$-code of type Q, $k=2$ and length $4n$, we obtain a new $\HFP$-code of type Q, $k=2$ and length $8n$ and we show that when the former code has maximum rank $r=2n$ then the constructed code of double length has also maximum rank.

\section{Hadamard Full Propelinear Co\-des}\label{sec:2}
We call $(C,\cdot)$ a propelinear Hadamard code if $C$ is both, a propelinear code and an Hadamard code.
In this section we introduce the concept of Hadamard full propelinear code and we show that it is equivalent to the well known concepts of Hadamard group, 2-cocyclic matrix and relative difference set.


\begin{definition}[\cite{rs}]
	An $\HFP$-code (Hadamard full propelinear code) is an Hadamard propelinear code $C$ such that for every $a\in C$, $a\not=\be, a\not=\bu$ the permutation $\pi_a$ has not any fixed coordinate and $\pi_\be=\pi_\bu=I$.
\end{definition}
\begin{lemma}\label{lemm:central}
	In an $\HFP$-code $(C,\cdot )$ the vector $\bu$ is a central involution in $C$.
\end{lemma}
\begin{proof}
	For any vector $a\in C$ we have $a\bu=a+\pi_a(\bu)=a+\bu=\bu+\pi_\bu(a)=\bu a$, so vector $\bu$ is central in $C$. Also $\bu^2=\bu+\pi_{\bu}(\bu)=\bu+\bu=\be$. \qed
\end{proof}

\begin{remark} \label{remark:1} ~\\
	Let $(C,\cdot )$ be an $\HFP$-code of length $4n$ and let $D_1$ be the set of codewords with a zero in the first coordinate. The cardinality of $D_1$ is $4n$ and we can use $D_1$ as an indexing set for the coordinates of the elements in $C$.
	We will say that $x\in D_1$ is indexing the $i$th coordinate when
	\begin{equation}\label{indexing}
	\mbox{$e_1=\pi_x(e_i)$,}
	\end{equation}
	where $e_i$ is defined in Introduction. Note that for any $i\in \{1,\ldots, 4n\}$ there is one and only one element $x\in D_1$ such that $e_1=\pi_x(e_i)$ (otherwise, code $C$ would not be full propelinear).
	
	The value of the coordinate indexed by $x\in D_1$ in a vector $a\in C$ is zero or one, depending on the value of the first coordinate of $\pi_x(a)$. As $x\in D_1$, the value of the first coordinate of $\pi_x(a)$ is the same as in $x+\pi_x(a)=xa$. Hence, let  $\delta_{(xa)}\in \{\be,\bu\}$ be such that $\delta_{(xa)}xa\in D_1$, then the value of the coordinate indexed by $x\in D_1$ in vector $a\in C$ is given by $\gamma_{(xa)}\in \Z_2$, where $\gamma_{(xa)}=0$ if and only if $\delta_{(xa)}=\be$. \qed
\end{remark}

A previous version of the next statement was proved in \cite{rs} but we include a new proof here.
\begin{proposition}\label{lemm:2.3}
	Let $(C,\cdot )$ be an $\HFP$-code of length $4n$ and let $D_1$ be the set of codewords with a zero in the first coordinate. Then the triple $(C,D_1,\bu)$ is a left Hadamard group.
\end{proposition}
\begin{proof}
	Since $C$ is an $\HFP$-code $\pi_x(e_1)\not=e_1$, except for $x\in\{\be,\bu\}$. We have that $|D_1\cap xD_1|=4n$ if and only if $x=\be$, and  $|D_1\cap xD_1|=0$ if and only if $x=\bu$. For the other cases $\pi_x(e_1)=e_k$, where $e_1\not=e_k$ and $xD_1$ is either $D_k$ or $\bu D_k$, depending on $x\in D_k$ or $x\notin D_k$, respectively, where $D_k$ is the set of codewords with a zero in the $k$th position. In any case $|D_1\cap xD_1|=2n$.
	This proves the conditions in \Cref{hg}. \qed
\end{proof}

Vice versa, after Proposition~\ref{lemm:2.3}, the next proposition shows that starting from a left Hadamard group $(G,D,u)$ we can construct an $\HFP$-code, isomorphic to $G$ as a group, with $D$ corresponding to the set of codewords with a zero in the first position and $u$ corresponding to the all ones vector. A previous version of this result could be seen in \cite{rs}.
\begin{proposition}\label{hgtohfp}
	Let $(G,*)$ be an Hadamard group of order $8n$ with $D$ as the prescribed Hadamard subset corresponding to a central involution $u$. Then we can construct an $\HFP$-code $C$, isomorphic to $G$ as a group. This group isomorphism $\sigma: G\rightarrow C$ is such that $\sigma(D)=D_1$ and $\sigma(u)=\bu$, where $D_1$ is the set of codewords of $C$ with a zero in the first position corresponds and $\bu$ is the all ones vector.
\end{proposition}
\begin{proof}
	We can assume the identity element $e$ of $G$ is in $D$, otherwise we take as $D$ the set $u*D$.
	From $G$ we can construct a $4n\times 4n$ Hadamard matrix $H$, where the columns are indexed by the elements of $D$.
	Order the columns in such a way that $e\in D$ is indexing the first column.
	For $a,b\in D$, the entry $(a,b)$ of $H$ is $0$ if $b*a\in D$ and $1$ if $b*a\notin D$. We will say that the entry $(a,b)$ of $H$ is $\gamma_{(b*a)}\in \Z_2$, where $\gamma_{(b*a)}=0$ if and only if $\delta_{(b*a)}=e$. The value of $\delta_{(b*a)}\in \{e,u\}$ is such that $\delta_{(b*a)}*b*a\in D$. For a fixed $a\in D$, the vector with entries $(a,b)$, for $b\in D$ will be the corresponding codeword $\sigma(a)$ in the code $C$ we are constructing. The first coordinate of codeword $\sigma(a)$ is zero.
	
	First of all we show that $H^T$ is an Hadamard matrix and so $H$ is an Hadamard matrix too. Take two columns of $H$, indexed by $b,c\in D$, respectively.
	These two columns have the same value in the row position corresponding to codeword $\sigma(a)$ if and only if $\delta_{(b*a)}=\delta_{(c*a)}$ or, the same, if and only if either $a\in b^{-1}*D \cap c^{-1}*D$ or $a*u \in b^{-1}*D \cap c^{-1}*D$. Hence, since $D$ is an Hadamard subset, there are $|b^{-1}*D \cap c^{-1}*D|=|c*b^{-1}*D \cap D|=2n$ positions where the two different columns indexed by $b,c\in D$ coincide. As $e\in D$ then $H$ is a normalized Hadamard matrix.
	
	Now, we construct a full propelinear code $C$. The codewords of $C$ are the rows (and the complements) of matrix $H$. For any $a\in G$, $\sigma(a)$ will be the corresponding row (or the complement) of $H$ constructed from $a$. Obviously, $\sigma(e)=\be$ and $\sigma(u)=\bu$.
	
	For any $\sigma(a)\in C$ its coordinates are indexed by the elements in $D$ and we define a map $\pi_{\sigma(a)}: C \longrightarrow \F^n$ by
	$$
	\pi_{\sigma(a)}(\sigma(x))=\sigma(a)+\sigma(a*x),\,\mbox{for any }\, \sigma(x)\in C,
	$$
	where the operation $+$ is the componentwise addition in $\F^n$.
	The map $\pi_{\sigma(a)}$ acts as a permutation on $D$. Specifically, the coordinate given by $b\in D$ is moved to the coordinate given by $\delta_{(b*a^{-1})}*b*a^{-1}\in D$.
	Indeed, take the element $x\in D$ to show that the value of the coordinate indexed by $b\in D$ in $\sigma(x)$ coincides with the value of the coordinate indexed by $\delta_{(b*a^{-1})}*b*a^{-1}\in D$ in $\pi_{\sigma(a)}(\sigma(x))$. The value of the coordinate indexed by $b$ in $\sigma(x)$ is $\gamma_{(b*x)}$. The value of the coordinate indexed by $\delta_{(b*a^{-1})}*b*a^{-1}$ in $\sigma(a)$ is $\gamma_{(b*a^{-1})}+\gamma_{b}=\gamma_{(b*a^{-1})}$ ($\gamma_{b}=0$ since $b\in D$). On the other hand, the value of the coordinate indexed by $\delta_{(b*a^{-1})}*b*a^{-1}$ in $\sigma(a*x)$ is $\gamma_{(b*a^{-1})}+\gamma_{(b*x)}$. So, we should check that  $\gamma_{(b*x)}= \gamma_{(b*a^{-1})}+ \gamma_{(b*a^{-1})}+\gamma_{(b*x)}$, which is obvious.
	Therefore, $\pi_{\sigma(a)}$ acts as a permutation on the set of coordinates.
	
	For any $\sigma(a), \sigma(b)\in C$ we define $\sigma(a)\sigma(b)=\sigma(a)+\pi_{\sigma(a)}(\sigma(b))=\sigma(a*b)$ which gives a propelinear structure on $C$. Indeed, the conditions in \Cref{def:propelinear} are fulfilled. The first one is straightforward. For the second condition we want to prove that, for all $\sigma(x),\sigma(y) \in C$, we have $\pi_{\sigma(x)}\pi_{\sigma(y)}=\pi_{\sigma(x)\sigma(y)}$. Take any $\sigma(z)\in C$ and compute $\pi_{\sigma(x)\sigma(y)}(\sigma(z))=\pi_{\sigma(x*y)}(\sigma(z))=\sigma(x*y)+\sigma((x*y)*z)= \sigma(x*y)+\sigma(x*(y*z))= \sigma(x) + \pi_{\sigma(x)}(\sigma(y))+\sigma(x)+\pi_{\sigma(x)}(\sigma(y*z))=\pi_{\sigma(x)}(\sigma(y)+\sigma(y*z))=\pi_{\sigma(x)}(\pi_{\sigma(y)}(\sigma(z)))$.
	This proves the statement.\qed
\end{proof}

\begin{remark}~\\
	These two last propositions prove the equivalence of both concepts, Hadamard groups $(G,D,u)$ and $\HFP$-codes. If we begin with an $\HFP$-code $(C,\cdot)$ we can consider the Hadamard group that such code gives rise to (\Cref{lemm:2.3}) and then, from the Hadamard group we can construct an $\HFP$-code which coincides with the initial one (\Cref{hgtohfp} and \Cref{remark:1}). Vice versa, starting from an Hadamard group $(G,D,u)$ we can construct an $\HFP$-code (which is isomorphic to $G$ as a group), with $D_1=\sigma(D)$ as the set of codewords with a zero in the first position and $\bu=\sigma(u)$ as the all ones vector. This $\HFP$-code coincides with the Hadamard group $(G,D,u)$ (\Cref{lemm:2.3}).
	\qed
\end{remark}

It would seem, on the surface, that only left Hadamard groups are in one-to-one correspondence with $\HFP$-codes. Next proposition shows that the correspondence is also for right Hadamard groups.

\begin{proposition}\label{prop:3}
	Let $(C,\cdot )$ be an $\HFP$-code of length $4n$ and let $D_1$ be the set of codewords with a zero in the first coordinate. Then the triple $(C,D_1,\bu)$ is a right Hadamard group.
\end{proposition}
\begin{proof}
	To show that $(C,D_1,\bu)$ is a right Hadamard group note that using all $a\in D_1$ the value of $\pi_a^{-1}(e_1)$ gives all $e_i$, for $i\in \{1,\ldots,4n\}$, otherwise the code would not be full propelinear.
	Indeed, if $\pi_a^{-1}(e_1)=\pi_b^{-1}(e_1)$, for $a,b\in D_1$, with $a\not= b$ then $\pi_{ab^{-1}}(e_1)=e_1$ which contradicts that $C$ is $\HFP$.
	Vectors in $D_1 \cap D_1x$ are those in $\{a+\pi_a(x) : a\in D_1\}$ with a zero in the first coordinate, which coincides with those in $\{\pi_a(x) : a\in D_1\}$ with a zero in the first coordinate. The value of the first position in $\pi_a(x)$ is the value on the $i$th position of vector $x$, knowing that $\pi_a^{-1}(e_1)=e_i$. When we take all $a\in D_1$ the value of the first position in $\pi_a(x)$ take all values given by all positions of vector $x\notin \{\be,\bu\}$, so half zeros and half ones. Hence, for $x\notin \{\be,\bu\}$ we have $|D_1 \cap D_1 x|=2n$. This proves the first condition in \Cref{hg}. The second condition is obvious.
	\qed
\end{proof}

\begin{corollary}\label{coro:invers}
	If $(G,D,u)$ is a left Hadamard group then $(G,D,u)$ is a right Hadamard group; $(G,D^{-1},u)$ is a left Hadamard group and $(G^{\op},D,u)$ is a left Hadamard group.
\end{corollary}
\begin{proof}
	Let $(G,D,u)$ be a left Hadamard group. From \Cref{prop:3,hgtohfp} we have that $(G,D,u)$ is also a right Hadamard group. Now, $(G^{\op},D,u)$ is a left Hadamard group and using the isomorphism $\phi:G^{\op} \rightarrow G$ such that $\phi(g)=g^{-1}$ we obtain that $(G,D^{-1},u)$ is also a left Hadamard group. \qed
\end{proof}

It is worth to mention that in \cite{itoremarks} it was proved  that if $(G,D,u)$ is a left Hadamard group then $(G,D^{-1},u)$ is also a left Hadamard group, considering the group ring of $G$ over the field of complex numbers.

Let $(G,D,u)$ be an Hadamard group. Following \Cref{hgtohfp} construct the associated $\HFP$-codes $C, E$ to $(G,D,u)$ and $(G^{\op},D,u)$, respectively. It is easy to see that the corresponding normalized Hadamard matrices of codes $C, E$ are transpose with one another. We will say that the $\HFP$-code $E$ is the \textit{transpose code} of $C$.

\begin{corollary}\label{coro:trans}
	Let $C$ be an $\HFP$-code of length $4n$ and $(C,D_1,\bu)$ the corresponding Hadamard group (\Cref{lemm:2.3}). The transpose $\HFP$-code of $C$ has $(C,D_1^{-1},\bu)$ as Hadamard group.
\end{corollary}
\begin{proof}
	Straightforward from the proof of \Cref{coro:invers}.
\end{proof}
\begin{proposition}[\cite{bbr}]
	Let $C$ be an $\HFP$-code of length $4n$. Set $\Pi =\{\pi_x : x\in C\}$. Then $\bu\in K(C)$ and  $\Pi$ is isomorphic to $C/\langle \bu \rangle$.
\end{proposition}
\begin{proof}
	The fact that $\bu\in K(C)$ is straightforward. The map $x\longrightarrow \pi_x$ is a group homomorphism from $C$ to $\Pi$. Since $C$ is full propelinear, the kernel of this homomorphism is $\langle \bu\rangle$. The statement is proven.\qed
\end{proof}

Next proposition, proved in \cite{rs}, extends a previous result by Ito in \cite{ito} where it was proved for a Sylow subgroup of $C$.
\begin{proposition}[\cite{rs}]\label{lemm:cyclic}
	Let $(C,\cdot )$ be an Hadamard propelinear code of length $4n$. Then $|C|=8n$, but it is not a cyclic group of order $8n$.
\end{proposition}

It is well known that there are five inequivalent Hadamard codes of length 16. One of them is linear, another is a $\Z_2\Z_4$-linear code \cite{prv2}, and the other three cannot be realized as $\Z_2\Z_4$-linear codes. However, one of those can be realized as a $\Z_2\Z_4Q_8$-code, more specifically, as a pure $Q_8$-code~\cite{riri}. As an easy example of $\HFP$-codes it was shown in \cite{rs} that the last two are full propelinear codes, which are not translation invariant. The group structure of these two propelinear codes correspond to a generalized quaternion group of order 32.

\section{Rank and Kernel of \boldmath{$\HFP$}-codes of type Q}\label{sec:3}
In this section we study the allowable values for the rank and for the dimension of the kernel of an $\HFP$-code of type Q. The main result is that for $\HFP$-codes of type Q and length $4n= 2^sn'$, where $n'$ is odd and $s=2$, we have rank $r=4n-1$ and dimension of the kernel $k=1$; and when $s>2$ we have rank $r\leq 2n$ and dimension of the kernel 1 or $2$.
\begin{proposition}\label{proj}
	Let $C$ be an Hadamard code of length $4n$ and $s\in K(C)$, $s\notin \{\be,\bu\}$. Then, the projection of $C$ onto $\Supp(s)$ is an Hadamard code of length $2n$.
\end{proposition}
\begin{proof}
	Without loss of generality we can write $s$ in the form $(s_1,s_2)$ where $s_1=(1,\ldots,1)$, $s_2=(0,\ldots,0)$.
	Let $v=(v_1,v_2), w=(w_1,w_2)$ be any two vectors in $C$. After projecting the $8n$ vectors of $C$ over $\Supp(s)$ we have $2n$ different vectors and their complements. Indeed, exactly the vectors $v=(v_1,v_2), w=(v_1,v_2+\bu)$ goes to the same vector in the projection over the support of $s$. Hence, we want to prove that if $v_1\not= w_1$ and $v_1\not= w_1+\bu$ then $d(v_1,w_1)=n$.
	
	Since $s\in K(C)$, for any $v\in C$ we have $s+v\in C$, so $(\bu+v_1,v_2)\in C$ and for any $w\in C$ such that $v\neq w, v\neq w+\bu, s+v\neq w, s+v\neq w+\bu$ we have $d(s+v,w)=2n$ and also $d(v,w)=2n$. Hence, $d(v_1+\bu,w_1)+d(v_2,w_2)=2n$ and $d(v_1,w_1)+d(v_2,w_2)=2n$.
	Therefore, $d(v_1+\bu,w_1)=d(v_1,w_1)$ and since $\bu+v_1$ is the complementary vector of $v_1$ we have $d(\bu+v_1,w_1)+d(v_1,w_1)= 2n$. Thus, $d(v_1+\bu,w_1)=d(v_1,w_1)= n$. This proves the statement. \qed
\end{proof}

Next three lemmas are well known. We include them without proof.
\begin{lemma}[\cite{bmrs}]\label{prop:3.2}
	Let $(C,\cdot)$ be a propelinear code of length $4n$.
	\begin{enumerate}[i)]
		\item  For $x\in C$ we have $x \in K(C)$ if and only if $\pi_x \in \Aut(C)$.
		\item The kernel $K(C)$ is a subgroup of $C$ and also a binary linear space.
		\item If $c\in C$ then $\pi_c\in Aut(K(C))$ and $c\cdot K(C)=c+K(C)$.
	\end{enumerate}
\end{lemma}

\begin{lemma}[\cite{prv1,prv3}]\label{odd}
	Let $C$ be a non linear Hadamard code of length $2^sn'$, where $n'$ is odd. The dimension of the kernel is $1\leq k\leq s-1$.
\end{lemma}
\begin{lemma}[\protect{\cite[Th. 2.4.1 and Th. 7.4.1]{ak}}] \label{aske}
	Let $C$ be an Hadamard code of length $4n=2^sn'$, where $n'$ is odd.
	\begin{enumerate}[(i)]
		\item If $s=2$ then $r = 4n-1$.
		\item If $s\geq 3$ then the rank of $C$ is $r\leq 2n$, with equality if $s=3$.
	\end{enumerate}
\end{lemma}

\begin{definition}
	Let $C$ be an $\HFP$-code of length $4n=2^sn'$ with $n'$ odd. We will say that $C$ is a code of type Q when it is a group of type Q~\cite{itoIII}
	\begin{eqnarray*}
		C=\langle\ba,\bb : \ba^{4n}=\be, \ba^{2n}=\bb^2, \bb^{-1}\ba\bb=\ba^{-1}\rangle.
	\end{eqnarray*}
\end{definition}

An $\HFP$-code $C$ of type Q contains only one involution  $\ba^{2n}=\bb^2=\bu$ which is central in $C$.

The interest of studying $\HFP$-codes of type Q is given by a conjecture of Ito stated in~\cite{itoIII} saying that for any length $4n$ it is possible to find an Hadamard group of type Q.

\begin{proposition}\label{rank-kernel}
	Let $C=\langle \ba,\bb \rangle$ be an $\HFP$-code of type Q and length $4n=2^sn'$ with $n'$ odd. Then
	\begin{enumerate}[(i)]
		\item Code $C$ is linear if and only if the two following conditions are satisfied: $4n=2^s$, $r=k=s+1$.
		\item If $C$ is not linear then $\ba\notin K(C)$.
		\item If $s=2$ then $r=4n-1$ and $k=1$.
	\end{enumerate}
\end{proposition}
\begin{proof}
	First item is straightforward.
	
	For the second item, deny the statement, so assume that $\ba\in K(C)$. Then $\langle \ba \rangle \subset K(C)$ and so $|C/K(C)|\leq 2$. We conclude that $C=K(C)$ and $C$ is a linear code or else $|K(C)|=2^sn'$, with $n'=1$, which contradicts \Cref{odd}.
	
	The third item is straightforward from \Cref{odd,aske}.
\end{proof}

\begin{lemma}\label{prop:vice}
	Let $(C,\cdot)$ be a propelinear code of length $4n$. To check that $C$ is an Hadamard code it is enough to check that $(C,\cdot)$ has $8n-2$ codewords (different from $\be, \bu$) and that the weight of all these codewords is $2n$.
\end{lemma}
\begin{proof}
	Let $x,y\in C$ and set $z=x^{-1}\cdot y\in C$. We have $y=x\cdot z=x+\pi_x(z)$ and then $d(x,y)=\wt(x+y)=\wt(\pi_x(z))=\wt(z)$. Hence, we see that the distance of any two codewords  can be computed through the weight of a codeword. When $x\not=y$ or $x\not= y+\bu$ then $z\notin \{\be, \bu\}$ and so we only require that $\wt(z)=2n$ for all $8n-2$ codewords different from $\be$, $\bu$. This proves the statement.\qed
\end{proof}

\begin{proposition}\label{prop:3.7}
	Let $C=\langle \ba,\bb \rangle$ be an $\HFP$-code of type Q and length $4n$. Then, up to equivalence, we can assume
	\begin{enumerate}[i)]
		\item $\pi_\ba=(1,2,\ldots,2n)(2n+1,2n+2,\ldots,4n)$,
		\item $\pi_\bb=(1,4n)(2,4n-1)\ldots(2n,2n+1)$,
		\item $\Pi = C/\langle \bu \rangle$ is the dihedral group $\D_{2n}$ of order $4n$,
	\end{enumerate}
\end{proposition}
\begin{proof}
	Take $\ba\in C$ and see that since it is of order $4n$ then the associated permutation $\pi_\ba$ is of order $2n$. Indeed, if for $r<2n$ we have $\pi_\ba^r=I$ then $\ba^{2r}=\ba^r+\pi_{\ba^r}(\ba^r)=\ba^r+\pi_\ba^r(\ba^r)=\be$, which contradicts that $\ba$ is of order $4n$. Therefore we can think of $\pi_\ba$ as two disjoint cycles of length $2n$, say for instance $\pi_\ba=(1,2,\ldots,2n)(2n+1,\ldots,4n)$. In the case of $\pi_\bb$, with an analogous argumentation as before we can say that it is of order two and so $\pi_\bb$ is a composition of $2n$ disjoint transpositions. Each one of the transpositions sends an element of the first part of $\{1,2,\ldots,4n\}$ to the second part and vice versa. Indeed, assume for instance that $\pi_\bb$ moves the first position to the $(2n+i)$th, where $i\leq 2n$, which is the same position that we obtain using $\pi_{\ba^{i-1}}$, so $\pi_{\bb^{-1}\ba^{i-1}}$ has a fixed point which contradicts that $C$ is full propelinear.
	Furthermore, if we assume that $\pi_{\bb}$ moves, for instance, the first position to the $i$th position in the second part of $\{1,2,\ldots,4n\}$  then $\pi_\bb$ is uniquely determined.
	Indeed, as $\pi_\bb\pi_\ba=\pi_\ba^{-1}\pi_\bb$ we have that $2n$ is moved to $1$ by $\pi_{\ba}$, $1$ is moved to $(2n+i)$ by $\pi_\bb$ and $(2n+i)$ is moved to $(2n+i+1)$ by $\pi_\ba$. Hence, $2n$ is moved to $(2n+i+1)$ by $\pi_\bb$, and so on. Hence, $\pi_\bb=(1,2n+i)(2n,2n+i+1)(2n-1,2n+i+2)\ldots (2, 2n+i-1)$.
	Items $i)$ and $ii)$ are proven.
	Item $iii)$ is straightforward from the definition of $\pi_\ba$ and $\pi_\bb$.\qed
\end{proof}

In \Cref{remark:1}, for a generic $\HFP$-code $C$, we showed that we can consider the coordinates of any vector indexed by the elements of $D_1$ (the set of codewords with a zero in the first coordinate) in such a way that for any $x\in D_1$ the indexed coordinate is the $i$th such that $e_1=\pi_x(e_i)$.

The value of the $i$th coordinate (indexed by $x\in D_1$) of an element $y\in C$ is given by $\gamma_{(xy)}\in \Z_2$, where $\gamma_{(xy)}=0$ if and only if $\delta_{(xy)}=\be$. The value of $\delta_{(xy)}\in \{\be,\bu\}$ is such that $\delta_{(xy)}xy\in D_1$.

Now, after \Cref{prop:3.7} it is important not only to have the coordinates indexed by the elements in $D_1$, but also to order the elements in $D_1$ in such a way that the permutations of elements of $C$ be in the form given in that proposition. Hence, fix $\be\in D_1$ as the index for the first coordinate. From \Cref{prop:3.7} the second coordinate should be indexed by $x$ such that $e_2=\pi_{\ba}(e_1)$ and $e_1=\pi_x(e_2)$ (see \cref{indexing}). Hence, $e_1=\pi_{x\ba}(e_1)$ which means that either $x\ba=\be$ or $x\ba=\bu$, so and $x=\ba^{-1}$, up to complement. Again, the third coordinate should be indexed by $y$ such that $e_3=\pi_{\ba^2}(e_1)$ and $e_1=\pi_y(e_3)$ (see \cref{indexing}). Hence, $e_1=\pi_{y\ba^2}(e_1)$ which means that either $y\ba^2=\be$ or $y\ba^2=\bu$, so $y=\ba^{-2}$, up to complement. Again and again we will obtain that the coordinates $1$st, $2$nd, $\ldots$, $2n$th are indexed by the elements in $D_1$ corresponding (up to complement) to $\be, \ba^{-1}, \ldots, \ba^{-2n+1}=\ba\bu$, respectively.

Now, considering permutation $\pi_{\bb}$ in \Cref{prop:3.7}, the $2n+1$th coordinate is indexed by $x$ such that $e_{2n+1}=\pi_{\bb}(e_{2n})$ and $e_1=\pi_x(e_{2n+1})$ (see \cref{indexing}). Also $\pi_{\ba}^{2n-1}(e_1)=e_{2n}$.
Hence, $e_1=\pi_x(e_{2n+1})=\pi_{x\bb}(e_{2n})=\pi_{x\bb\ba^{2n-1}}(e_1)$ which means that $x\bb\ba^{2n-1}=\be$ and so $x=\ba^{-(2n-1)}\bb^{-1}$ or $x=\ba\bb$, up to complement. And again, with the same argumentation, we obtain that the coordinates $(2n+1)$th, $(2n+2)$th, $\ldots$, $4n$th are indexed by the elements in $D_1$ corresponding (up to complement) to $\ba\bb, \ba^{2}\bb, \ldots, \bb$, respectively.

Summarizing, we have the following remark which will be used throughout the paper.
\begin{remark}\label{order}~\\
	Let $C=\langle \ba,\bb \rangle$ be an $\HFP$-code of type Q and length $4n$. The coordinates of the elements in $C$ are indexed by the elements in $D_1$ (the set of vectors with a zero in the first coordinate) in such a way that the $i$th coordinate is indexed by $b\in D_1$ such that $e_1=\pi_b(e_i)$. We have that the coordinates $1$st, $2$nd, $\ldots$, $2n$th, $(2n+1)$th, $(2n+2)$th, $\ldots$, $4n$th are indexed by the elements in $D_1$ corresponding (up to complement) to $\be$, $\ba^{-1}$, \ldots, $\ba$, $\ba\bb, \ba^{2}\bb, \ldots, \bb$, respectively.
	
	Also, we have that the value of the coordinate indexed by $x\in D_1$ in vector $y\in D_1$ is given by $\gamma_{(xy)}\in \Z_2$, where $\gamma_{(xy)}=0$ if and only if $\delta_{(xy)}=\be$.
	The value of $\delta_{(xy)}\in \{\be,\bu\}$ is such that $\delta_{(xy)}xy\in D_1$.
	The Hadamard matrix $H$ constructed using as rows the vectors $y\in D_1$ is a normalized Hadamard matrix.
	In general, for any $x,y\in C$ we can say that the value of the coordinate indexed by $x$ (or the complement when $x\notin D_1$) in vector $y$ is given by
	\begin{equation}\label{eq:gammas}
	\gamma_{x}+\gamma_{y}+\gamma_{(xy)}\in \Z_2.
	\end{equation}
	
	Analogously, from \Cref{coro:trans} the code $C^T$ is an $\HFP$-code of type Q with corresponding Hadamard group $(C,D_1^{-1},\bu)$. The coordinates of elements in $C^T$ are indexed by the elements in $D_1^{-1}$ and, with the same argumentation as for code $C$, we will take the coordinates ordered in such a way that the permutations associated to elements in $C^T$ coincides with those \Cref{prop:3.7}. The order is given by: $\be$, $\ba$, \ldots, $\ba^{2n-1}$, $\ba\bb, \ba^{2}\bb, \ldots, \ba^{2n}\bb$.
	\qed
\end{remark}

Note that  both Hadamard groups $(C,D_1,\bu)$ and $(C,D_1^{-1},\bu)$ have the same elements and the same group structure. However, the binary representation of these elements, which depends on the Hadamard subset, could be different (in the first case we are talking about the rows of $H$ and in the second case about the columns).

Before going to the next proposition  we will take the elements of $C$ in a polynomial way, so the nonzero coefficients of $(a_1(x),a_2(x))$ are those in $\Supp(\ba)$, where $a_1(x), a_2(x)\in \F[x]$ have degree at most $2n-1$.
Permutation $\pi_a$ corresponds to multiplying by $x$ modulo $x^{2n}-1$ and so we have the correspondence $\ba^i \leftrightarrow \big((1+x+x^2+\ldots +x^{i-1})a_1(x),(1+x+x^2+\ldots +x^{i-1})a_2(x)\big)=\big(\frac{x^i-1}{x-1}a_1(x), \frac{x^i-1}{x-1}a_2(x)\big)$. Analogously, the nonzero coefficients of $(b_1(x),b_2(x))$ are those in $\Supp(\bb)$, where $b_1(x), b_2(x)\in \F[x]$ have degree at most $2n-1$.

Given a polynomial $p(x)$ of degree at the most $2n-1$, we will call $\varphi_1(p(x))=x^{2n-1}p(\frac{1}{x})$ and note that $\varphi_1(p(x))$ is not exactly the reciprocal polynomial of $p(x)$ but coincides with it in the case the degree of $p(x)$ is $2n-1$.

Denote by $u(x)$ the polynomial with all the coefficients 1 and by $e(x)$ the zero polynomial.

\begin{proposition}\label{prop:3.8}
	Let $C=\langle \ba,\bb \rangle$ be an $\HFP$-code of type Q and length $4n$. Then, if we know the value of $\ba$ then we can compute $\bb$ in a unique way, up to complement.
\end{proposition}
\begin{proof}
	We will take the elements in $C$ in a polynomial way, so $(a_1(x),a_2(x))$ is the polynomial representation of $\ba$.
	Permutation $\pi_\bb$ takes $(a_1(x),a_2(x))$ to $(\varphi_1(a_1(x)),\varphi_1(a_2(x)))$.
	
	From $\ba\bb=\bb\ba^{-1}$ we can write $\bb+\pi_\ba(\bb)=\ba+\pi_\bb(\ba^{-1})=\ba+\pi_{\ba\bb}(\ba)$ (indeed, $\ba^{-1}=\pi_{\ba^{-1}}(\ba)$). Vector $\bb$ and $\pi_\ba(\bb)$ have the same parity, so $\ba+\pi_{\ba\bb}(\ba)$ has even parity. Using polynomials we have $(x+1)b_i(x)=a_i(x)+x\varphi_1(a_{i'}(x)) \pmod{x^{2n}-1}$, for $i,i'\in\{1,2\}$, $i\not=i'$.
	Polynomials $a_i(x)+x\varphi_1(a_{i'}(x)) \pmod{x^{2n}-1}$ are multiples of $x+1$ (indeed, $\ba+\pi_{\ba\bb}(\ba)$ has even parity). Hence, $b_i(x)=\frac{a_i(x)+x\varphi_1(a_{i'}(x)) \pmod{x^{2n}-1}}{x-1}$ for $i,i'\in\{1,2\}$, $i\not=i'$. Note that, as $x-1$ divides $x^{2n}-1$, the solution of the equation giving $b_i(x)$ is unique up to complement. \qed
\end{proof}

A rough bound for $r$, depending on $k$, is established in the following lemma.
\begin{lemma}\label{brk}
	Let $C$ be an Hadamard code of length $2^sn'$, where $n'$ is odd. The rank $r$ of $C$ fulfills $r\leq \frac{2^{s+1}n'}{2^k}+k-1$, where $k$ is the dimension of the kernel.
\end{lemma}
\begin{proof}
	Let $k$ be the dimension of the the kernel $K(C)$, then we can see $C$ as a disjoin union of, at the most, $\frac{2^{s+1}n'}{2^k}$ cosets of $K(C)$. Hence, the rank of $C$ will be, at the most, $r\leq \frac{2^{s+1}n'}{2^k}-1+k$.\qed
\end{proof}

\begin{lemma}\label{lemm:4.1}
	Let $C=\langle \ba, \bb\rangle$ be an $\HFP$-code of type Q and length $4n$. Let $h$ be a divisor of $n$. Then $\pi_\ba^h(\ba^n)\not= \ba^n$.
\end{lemma}
\begin{proof}
	Indeed, assume the contrary, so $\pi_\ba^h(\ba^n)=\ba^n$, where $n=hh'$. Hence, $\bu=\ba^{2n}=\ba^n+\pi_\ba^{n}(\ba^n)=\ba^n+(\pi_\ba^h)^{h'}(\ba^n)=\ba^n+\ba^n=\be$, which is impossible.\qed
\end{proof}

Next theorem is one of the main results in the paper. It summarizes the values of the rank and dimension of the kernel for $\HFP$-codes of type Q.
\begin{theorem}\label{rk}
	Let $C=\langle \ba,\bb \rangle$ be a non linear $\HFP$-code of type Q and length $4n=2^sn'$, where $n'$ is odd. Let $r,k$ be the rank and the dimension of the kernel of $C$, respectively.
	\begin{enumerate}
		\item If $s=2$ then $C$ is a full rank code, so $r=4n-1$ and $k=1$.
		\item If $s=3$ then $r=2n$ and $k\in\{1,2\}$.
		\item If $s>3$ then $r\leq 2n$ and $k\in\{1,2\}$.
	\end{enumerate}
\end{theorem}

\begin{proof}
	The first and second items, as far as the rank is concerned, comes from  Proposition~\ref{rank-kernel} and \Cref{aske}. For the dimension of the kernel in the second item we use Lemma~\ref{brk} and \Cref{odd}.
	
	For the third item, we begin by showing that $\ba^n \notin K(C)$. Assume the contrary. From \Cref{prop:3.2}, $\pi_\ba \in \Aut(K(C))$ and also, since $K(C)$ is a linear space and $\pi_\ba$ is a linear morphism, $\pi_\ba$ is a linear isomorphism of $K(C)$. We have $|K(C)|=2^k$. The amount of available values for $\pi_\ba(\ba^n)$ is upper bounded by $2^k-2$ so, if $i$ is the smallest index $i\leq 2n$ such that $\pi_\ba^i(\ba^n)=\ba^n$ then (\Cref{odd}),
	\begin{equation}\label{eq:3}
	i\leq 2^k-2 \leq 2^{s-1}-2.
	\end{equation}
	We know that $\bu=\ba^{2n}=\ba^n+\pi_{\ba^n}(\ba^n)$. We have $\pi_\ba^n(\ba^n)=\ba^n+\bu$ and $\pi_\ba^{2n}(\ba^n)=\ba^n$. Set $d=\gcd(i,n)$ with $d=\lambda i+\mu n$, for some integers $\lambda, \mu$. Compute $\pi_\ba^d(\ba^n)= \ba^n +\delta \bu$, where $\delta$ has the value $0,1$, depending on the parity of $\mu$ is either even or odd, respectively. Hence, $2d\geq i$ and $d$ is a proper divisor of $i$ (otherwise, $i$ would be a divisor of $n$, which is impossible from \Cref{lemm:4.1}) and therefore $2d=i$.\\
	Hence, $d$ is a divisor of $n$ and $2d=i$ is not a divisor of $n$. As $n=2^{s-2}n'$ we have $d=2^{s-2}n^*$, where $n^*\,|\, n'$. Finally, $i=2d=2^{s-1}n^*\geq 2^{s-1}$, which contradicts (\ref{eq:3}). This proves that $\ba^n \notin K(C))$.
	
	The order of elements in $K(C)$ has to be a power of two (indeed, the order of the group $K(C)$ is a power of two). Thus, if $\ba^i\in K(C)$ then take $j=\gcd(i,4n)$ and note that also $\ba^j \in K(C)$, where $j$ divides $4n$ . Therefore, for some $\nu$, $2^\nu j =4n$. Hence, as $\ba^n\notin K(C)$ we should have $j\in \{2n,4n\}$ and also $i\in\{2n,4n\}$. Therefore,  the elements in $K(C)$ different from $\be, \bu$ should be of the form $\ba^i\bb$. But if two of them, say $\ba^i\bb, \ba^j\bb$, with $i\not= j$ and $i\not=2n+j$, are in $K(C)$ then also $\ba^{i-j}\in K(C)$, which is impossible.
	
	Finally, the kernel is generated, at the most, by only one element different from $\bu$, say it is $\bkappa=\ba^{\iota}\bb$, for some $\iota\in \{0,1,\ldots,2n-1\}$. The dimension of the kernel is 1 or $2$. For the rank, the statement comes from \Cref{aske}.\qed
\end{proof}

\section{\boldmath{$\HFP$}-codes of type Q. The transpose}\label{sec:4}

Next step is to go further with the specific $\HFP$-codes of type Q and length $4n=2^sn'$, where $n'$ is odd. We show that the transpose of an $\HFP$-code of type Q and dimension of the kernel $k=2$ has always dimension of the kernel equal to 1.

\begin{proposition}\label{ker}
	Let $C=\langle \ba,\bb \rangle$ be an $\HFP$-code of type Q and length $4n$.  Assume that the permutations associated to the elements $\ba,\bb$ are those in \Cref{prop:3.7}.
	If the dimension of the
	kernel is $k=2$ then the vector $\bkappa$ in the kernel, different from $\be,\bu$, with a zero in the first coordinate is
	\begin{equation*}
	\bkappa=(\bv||\bw),\, \mbox{where }\, \bv=(0,1,0,1,\ldots,0,1)\, \mbox{and either}\,\, \bw=\bv \, \mbox{or}\,\, \bw=\bv+\bu.
	\end{equation*}
\end{proposition}
\begin{proof}
	Since the dimension of the kernel is 2, there exists an element $\bkappa\in C$, different from $\bu$, belonging to the kernel.
	Since $\bkappa^2$ is also in the kernel and it is different from $\bkappa$ or $\bkappa\bu$ we have $\bkappa^2=\bu$ or $\bkappa^2=\be$.
	Then, the associated permutation $\pi_{\bkappa}$ should be such that $\pi_{\bkappa}^2=I$ and so either $\bkappa=\ba^n$ or $\bkappa=\ba^i\bb$, for some $i\in \{0,\ldots,2n-1\}$.
	
	From \Cref{prop:3.2} we have $\pi_\ba\in \Aut(K(C))$ and so $\pi_\ba(\bkappa)\in\{\bkappa,\bkappa\bu\}$.
	As $n$ is even we have $\pi_\ba^n(\bkappa)=\bkappa$ so, from \Cref{lemm:4.1}, $\bkappa=\ba^i\bb$, for some $i\in\{0,1,\ldots,2n-1\}$. Again, from \Cref{prop:3.2}, we have $\pi_\ba(\bkappa)=\bkappa+\epsilon$,
	where $\epsilon\in\{\be,\bu\}$. Hence, we are in one of the following cases:
	\begin{enumerate}[i)]
		\item $\bkappa =(1,1,\ldots,1||0,0,\ldots,0)+\epsilon$,
		\item $\bkappa=(\bv||\bw)$, where $\bv,\bw\in  \{(1,0,1,0,\ldots, 1,0),(0,1,0,1,\ldots,0,1)\}$.
	\end{enumerate}
	Note that the first item could not happen. Indeed, by projecting the vectors to the  first $2n$ coordinates, we obtain an Hadamard code of length $2n$ (\Cref{proj}).
	This code is a cyclic group of order $4n$ generated by the projection of $\ba$, which contradicts \Cref{lemm:cyclic}.
	This proves the statement. \qed
\end{proof}

Summarizing the above results, we conclude with the following Proposition.

\begin{proposition}\label{prop:typeR}
	Let $C=\langle \ba,\bb \rangle$ be an $\HFP$-code of type Q, length $4n$ and dimension of the kernel $k=2$.  Let $\bkappa$ be the vector in the
	kernel, different from e,u, with a zero in the first coordinate.
	Then, up to equivalence,
	\begin{enumerate}[i)]
		\item $\pi_\ba=(1,2,\ldots,2n)(2n+1,2n+2,\ldots,4n)$,
		\item $\pi_\bb=(1,4n)(2,4n-1)\cdots(2n,2n+1)$,
		\item $\bkappa=\ba^\iota\bb$, up to complement, for some $\iota\in \{0,\ldots,2n-1\}$ and either $\bkappa=(\bv||\bv)\in K(C)$, when $\iota$ is even, or $\bkappa=(\bv||\bw)\in K(C)$, when $\iota$ is odd, where $\bv=(0,1\ldots,0,1)$, $\bw=\bv+\bu$.
		\item $\ba=(\ba_1||\ba_2)$ where, in polynomial way, $a_2(x)=x^{\iota+1}\varphi_1(a_{1}(x))+u(x)$, $u(x)$ is the polynomial of degree $2n-1$ with all the coefficients 1, and $\iota$ is the exponent in the above item $iii)$.
	\end{enumerate}
\end{proposition}
\begin{proof}
	Items \textit{i), ii)} are straightforward from the previous results.
	
	For item $iii)$, note that the two non trivial elements in the kernel are $\bkappa=\ba^\iota\bb$ and $\bkappa\bu$, for some index $\iota\in \{0,\ldots, 2n-1\}$.
	Take the vector $\bkappa=\ba^\iota\bb$ (the first coordinate may be zero or one). We know that $\bkappa^2=\bu$ and so $\bu=\bkappa+\pi_{\bkappa}(\bkappa)=\bkappa+\pi_{\ba^{\iota}}\pi_{\bb}(\bkappa)$. Hence, if $\iota$ is even then, up to complement, $\kappa=(\bv||\bv)$ and if $\iota$ is odd then $\kappa=(\bv||\bw)$, where $\bv =(0,1,\ldots, 0,1), \bw= (1,0,\ldots,1,0)\}$.
	
	Item $iv)$ comes from the same argumentation as in \Cref{prop:3.8}, where instead of $\bb$ we use $\ba^\iota\bb$.
	Indeed, from $\bkappa\ba^{-1}=\ba\bkappa$ we obtain $\bkappa+\pi_{\bkappa}(\ba^{-1})=\ba+\pi_\ba(\bkappa)$, so $\bkappa+\pi_\ba(\bkappa)=\ba+\pi_{\bkappa}(\ba^{-1})=\ba+\pi_{\ba^\iota}\pi_{\bb}(\ba^{-1})=\ba+\pi_{\ba^\iota}\pi_{\bb}\pi_{\ba^{-1}}(\ba)= \ba+\pi_{\ba^{\iota+1}}\pi_{\bb}(\ba)$. But, $\bu=\bkappa +\pi_{\ba}(\bkappa)$ so, in a polynomial way, $u(x)=a_1(x)+x^{\iota+1}\varphi_1(a_{2}(x))$ and  $a_1(x)=x^{\iota+1}\varphi_1(a_{2}(x))+u(x)$. \\
	Analogously,  $a_2(x)=x^{\iota+1}\varphi_1(a_{1}(x))+u(x)$. This proves the statement.\qed
\end{proof}

\begin{theorem}\label{prop:trans}
	Let  $C=\langle \ba,\bb \rangle$ be an $\HFP$-code of type Q and length $4n$ and dimension of the kernel $k=2$.
	Then, the dimension of the kernel of the transpose $\HFP$-code is 1.
\end{theorem}
\begin{proof}
	Let $H$ be a normalized Hadamard matrix where the rows are elements of $C$ and, from \Cref{order}, the coordinates of these elements are indexed by the elements in $D_1$ (the set of vectors with a
	zero in the first coordinate) with the order given by: $\be$, $\ba^{-1}$, \ldots, $\ba^{-(2n-1)}$, $\ba\bb, \ba^{2}\bb, \ldots, \ba^{2n}\bb$.
	The columns of $H$ are elements of $C^T$ (\Cref{coro:trans}) and their coordinates are indexed by the elements in $D_1^{-1}$ with the corresponding order:  $\be$, $\ba$, \ldots, $\ba^{(2n-1)}$, $\ba \bb, \ba^{2} \bb, \ldots, \ba^{2n} \bb$.
	
	Now, we assume that the dimension of the kernel is 2 in both codes $C$ and $C^T$ so we can use the results in \Cref{prop:typeR} for both codes.
	We prove that this assumption leads to us to a contradiction.
	
	Let $\bkappa_1= \ba^{\iota}\bb  \in K(C)$ (respectively, $\bkappa_2=\ba^{\bar{\iota}}\bb \in K(C^T)$) the vector in the kernel of $C$ (respectively, $C^T$) different from $\be,\bu$ and with a zero in the first coordinate.
	
	First of all we are going to see that the parity of $\iota, \bar{\iota}$ is different from each other.
	Deny the proposal and assume that $\iota, \bar{\iota}$ have the same parity.
	So, the row indexed by $\ba^\iota\bb$ and the column indexed by $\ba^{\bar\iota}\bb$ are equal. Consider the row indexed by $\ba^\iota\bb$ and compute the value of its coordinates, indexed by all different $\ba^j\bb$, for $j\in \{0,\cdots,2n-1\}$. From \cref{eq:gammas} these values are given by $\gamma_{\ba^\iota\bb}+\gamma_{\ba^j\bb}+\gamma_{(\ba^j\bb)(\ba^\iota\bb)}=
	\gamma_{\ba^\iota\bb}+\gamma_{\ba^j\bb}+\gamma_{\ba^{j-\iota}}+\gamma_{\bu}$. Row vectors $\ba^j\bb$, in the coordinate indexed by $\ba^{\bar\iota}\bb$, have the following value: $\gamma_{\ba^{\bar\iota}\bb}+\gamma_{\ba^j\bb}+\gamma_{(\ba^{\bar{\iota}}\bb)(\ba^j\bb)}=
	\gamma_{\ba^{\bar\iota}\bb}+\gamma_{\ba^j\bb}+\gamma_{\ba^{\bar{\iota}-j}}+\gamma_{\bu}$.
	
	Assuming that the row indexed by $\ba^\iota\bb$ coincides with the column indexed by $\ba^{\bar\iota}\bb$ we would have that these two previous values coincides, so:
	\begin{equation}\label{eq:gamma1}
	\gamma_{\ba^{j-\iota}}+\gamma_{\ba^{\bar{\iota}-j}}=\gamma_{\ba^{\bar\iota}\bb}+\gamma_{\ba^\iota\bb},\,\,\,
	\mbox{for all $j\in \{0,\ldots, 2n-1\}$}.
	\end{equation}
	Now, as $\iota, \bar{\iota}$ have the same parity we can take $j =\frac{\iota+{\bar\iota}}{2}$ in \cref{eq:gamma1} to obtain \begin{equation}\label{eq:con}
	\gamma_{\ba^{\bar\iota}\bb}+\gamma_{\ba^\iota\bb}=0.
	\end{equation}
	If $\iota+{\bar\iota}\leq 2n$ (respectively, $\iota+{\bar\iota}\geq 2n$) then taking $j=\frac{\iota+{\bar\iota}+2n}{2}$ in \cref{eq:gamma1} (respectively, $j=\frac{\iota+{\bar\iota}-2n}{2}$) we obtain $\ba^{j-\iota}= \ba^{{\bar \iota}-j}\ba^{2n}$ and so $\gamma_{\ba^{\bar\iota}\bb}+\gamma_{\ba^\iota\bb}=1$, which contradicts \cref{eq:con}.
	Hence, we conclude that it could not happen that the row indexed by $\ba^\iota\bb$ and the column indexed by $\ba^{\bar\iota}\bb$ coincides. Therefore, the parity of $\iota, \bar{\iota}$ is different from each other.

	Since $\bkappa_1$ belongs to the kernel of $C$ the vectors in $C$ with a zero in the first coordinate can be separated into two disjoint classes, the class $A_1=\{\gamma_{\ba^j}\ba^i\,:\, 0 \leq j \leq 2n-1\}$ and the class $A_2=\ba^\iota \bb + A_1$. Both classes are defined up to complement. Analogously for $C^T$.
	
	Without loss of generality, we can assume that $\iota$ is odd and ${\bar \iota}$ is even.
	From \Cref{prop:typeR}, in matrix $H$ 
	row vectors $\ba^j$, for even $j$, have the same coordinates in both halves and also vectors $\ba^j\bb$, for odd $j$, have the same coordinates in both halves.
	Therefore, projecting each of these vectors over the first half part we obtain $2n$ vectors of length $2n$ and weight $n$. Furthermore, the distance between them is also $n$ and so they form an Hadamard matrix $E$.
	One of the rows of $E$ is $\bkappa_1^{(p)}$, the projection of vector $\bkappa_1 = \ba^{\iota}\bb$ over the first half of coordinates. Vector $\bkappa_1^{(p)}$ is in the kernel of the code given by $E$. Indeed, from \Cref{prop:3.2} we have $\ba^j + K(C)= \ba^j\cdot K(C)$ and so $\ba^j + \bkappa_1 =\ba^j\ba^{\iota}\bb$ (up to complement). Hence, for even $j$, $(\ba^j)^{(p)} + \bkappa_1^{(p)}= (\ba^{j+\iota}\bb)^{(p)}$, which is a row of matrix $E$ (up to complement).
	
	Now, we repeat the operation using column vectors of $E$. Column vectors corresponding to odd columns have the same coordinates in both halves. Hence, projecting each of these vectors over the first half part we obtain $n$ vectors of length $n$ and weight $n/2$. The distance between them is also $n/2$ and so they form an Hadamard matrix $F$.
	The rows of $F$ are the projections of $\ba^j$, for even $j\in \{0,2, \ldots, 2n-2\}$ over the coordinates indexed by $\ba^j$, for even $j\in \{2, \ldots, 2n\}$. 
	If $\ba^j=(a_1^{(j)}, a_2^{(j)},\ldots, a_{4n}^{(j)})$ is a row in $H$, then $\ba_p^j=(a_1^{(j)}, a_3^{(j)},\ldots, a_{2n-1}^{(j)})$ is the corresponding row in $F$.

	The group structure of the given code $\HFP$-code $C$ is $\langle \ba,\bb \rangle$ from where $\langle \ba^2 \rangle$ is a cyclic subgroup with $2n$ elements. The associated permutation to $\pi_{\ba}$ is a cyclic shift to the right, so the associated permutations to elements $\ba^2$ are well defined acting over the set of even coordinates. Hence, the projection of elements in $\langle \ba^2 \rangle$ over the even coordinates is a cyclic propelinear code $C^\prime$ of length $n$ and $|C^\prime|=2n$. Since the elements in $C^\prime$ are exactly those in $F$ (up to complement) we conclude that $C^\prime$ is a cyclic $\HFP$-code of length $n$ and $|C^\prime|=2n$ which, from \Cref{lemm:cyclic}, does not exist.
	The statement is proven.\qed
\end{proof}

\section{\boldmath{$\HFP$}-codes. Constructions}\label{sec:5}

In this section we start from an $\HFP$-code of type Q, length $4n$ and dimension of the kernel 2 and we construct a new $\HFP$-code with double length $8n$ and the same dimension of the kernel $k=2$. We also show that when the initial code has maximum rank $r=2n$ the obtained code of length $8n$ has also maximum rank $4n$.

Throughout the section we will take $C=\langle \ba,\bb \rangle$ as an $\HFP$-code of type Q, length $4n=2^sn'$, $s \geq 3$, $n'$ odd and $k=2$. The kernel is $K(C)=\langle \bu,\bkappa\rangle$ where, up to complement, $\bkappa=\ba^\iota\bb$ for some $\iota\in \{0,\ldots,2n-1\}$ (see \Cref{prop:typeR}).

As we already said, given a polynomial $p(x)$ of degree at most $2n-1$ the polynomial $x^{2n-1}p(1/x)$ will be denoted by $\varphi_1(p(x))$.  Now, given a polynomial $p(x)$ of degree at most $4n-1$ the polynomial $x^{4n-1}p(1/x)$ will be denoted by $\varphi_2(p(x))$.

We begin with an example and two technical lemmas. The example is to show that there are $\HFP$-codes of type Q, length $4n$ and dimension of the kernel 2. The lemmas are about the greatest common divisor of polynomials, which will help to compute some ranks.

\begin{example} The code $C=\langle \ba,\bb \rangle$ is an $\HFP$-code of type Q, length $4n=2^sn'=24$ (so $s=3, n'=3$). The rank is $r=12$, and dimension of the kernel is $k=2$ (the kernel is $K(C)=\langle \bu,\bkappa \rangle$, where $\kappa=\ba^{11}\bb$). The generators are
	\begin{equation*}
	\begin{split}
	\ba= &(1, 1, 1, 1, 1, 1, 0, 1, 1, 0, 1, 0\,|| \, 1, 0, 1, 0, 0, 1, 0, 0, 0, 0, 0, 0);\\
	\bb= &(0,1,0,1,0,1,1,1,0,0,0,0\,|| \,1,1,1,1,0,0,0,1,0,1,0,1),
	\end{split}
	\end{equation*}
\end{example}
and the permutations associated to each codeword are given by \Cref{prop:typeR}.

The transpose of this code has the same rank, but dimension of the kernel equal to 1.

\begin{lemma}\label{prop:double1}
	Let  $C=\langle \ba,\bb \rangle$ be an $\HFP$-code of type Q, length $4n$, $k=2$ and set $K(C)=\langle \bu,\bkappa\rangle$, where $\bkappa=\ba^\iota\bb$ (up to complement) for an specific $\iota \in \{0,\ldots,2n-1\}$.
	
	If $\Gcd(a_1(x)+x^{\iota+1}\varphi_1(a_{1}(x)),x^{2n}-1)=x-1$ then $\rank(C)= 2n$, where $(a_1(x),a_2(x))$ is the polynomial representation for $\ba\in C$.
\end{lemma}
\begin{proof}
	From \Cref{prop:typeR}, the element $\ba=(a_1,a_2)\in C$ could be written, in a polynomial way, as $(a_1(x),x^{\iota+1}\varphi_1(a_{1}(x))+u(x))$.
	Since $\ba^{2n}=\bu$, the weight of $a_1$ is odd (indeed, $\ba^{2n}=\ba+\pi_{\ba}(\ba)+\pi_{\ba^2}(\ba)+\ldots \pi_{\ba^{2n-1}}(\ba)$ and so this means that in each coordinate of $\ba^{2n}$ there is the addition of all coordinates of $\ba$) and so $a_1(x)+x^{\iota+1}\varphi_1(a_{1}(x))+u(x)$ is a multiple of  $x-1$ (the weight of $u(x)$ is $2n$). Therefore, if $\Gcd(a_1(x)+x^{\iota+1}\varphi_1(a_{1}(x))+u(x), x^{2n}-1)=x-1$ then $\Gcd(a_1(x),a_1(x)+x^{\iota+1}\varphi_1(a_{1}(x))+u(x), x^{2n}-1)=1$. Polynomial $u(x)$ divides $x^{2n}-1$ and so $\Gcd(a_1(x),a_1(x)+x^{\iota+1}\varphi_1(a_{1}(x)), x^{2n}-1)=1$. Now, the dicyclic code generated by $a(x)=(a_1(x),a_2(x))$ has rank $2n$.
	
	On the other hand, the linear span of $C$, using polynomials, is linearly generated by
	\begin{equation}\label{eq:4}
	\{a(x), xa(x),\ldots, x^{2n-1}a(x), b(x), xb(x), \ldots, x^{2n-1}b(x)\},
	\end{equation}
	and, since the polynomials associated to $\bkappa$ are $(\kappa_1(x),\kappa_2(x))$, where $\kappa_i(x)= 1+x^2+\ldots +x^{2n-2}$, up to complement, and $x\kappa_i(x) =\kappa_i(x)+u(x)$ then \cref{eq:4} is simplified to
	\begin{eqnarray}
	\{ a(x), xa(x),\ldots, x^{2n-1}a(x), \kappa(x)\}.
	\end{eqnarray}
	So, the rank $r$ of the code $C$ is $2n\leq r\leq 2n+1$ and, from \Cref{aske} the rank is upper bounded by $2n$. Hence $\rank(C)=2n$. The statement is proven.\qed
\end{proof}

\begin{lemma}\label{prop:double2}
	Let  $C=\langle \ba,\bb \rangle$ be an $\HFP$-code of type Q, length $4n$, $k=2$ and set $K(C)=\langle \bu,\bkappa\rangle$, where $\bkappa=\ba^\iota\bb$ for an specific $\iota \in \{0,\ldots,2n-1\}$.
	
	If $\gcd(a_1(x^2)+x\kappa_1(x^2)+x^{2\iota+1}(a_1(x^2)+x\varphi_2(\kappa_{1}(x^2)), x^{4n}-1)\not= x-1$ then $\gcd(a_1(x)+x^{\iota+1}\varphi_1(a_{1}(x)),x^{2n}-1)\not= x-1$.
\end{lemma}
\begin{proof}
	If $p(x)$ is any polynomial of degree at the most $2n-1$ then we have $\varphi_2(p(x^2))= x(\varphi_1(p(x^2)))$.\\
	Indeed, if $p(x)=\sum_{i=0}^{2n-1}p_ix^i$ then $p(x^2)=\sum_{i=0}^{2n-1}p_{i}x^{2i}$ and  $\varphi_2(p(x^2))=\sum_{i=0}^{2n-1}p_ix^{4n-1-2i}$.
	On the other hand, $\varphi_1(p(x))=\sum_{i=0}^{2n-1}p_{i}x^{2n-1-i}$ and \\
	$x\varphi_1(p(x^2))=\sum_{i=0}^{2n-1}p_{i}x^{4n-1-2i}$.
	
	We have
	\begin{equation}\label{eq:33}
	\begin{split}
	&a_1(x^2)+x\kappa_1(x^2)+x^{2\iota+1}(a_1(x^2)+x\varphi_2(\kappa_1(x^2)) \\
	= &a_1(x^2)+x\kappa_1(x^2)+x^{2\iota+1}\varphi_2(a_1(x^2))+ x^{2\iota+1}x^{4n-1}\varphi_2(\kappa_1(x^2))\\
	= &a_1(x^2)+x^{2\iota+1}\varphi_2(a_1(x^2))+x\kappa_1(x^2)+ x^{2\iota} \varphi_2(\kappa_1(x^2))\\ =&a_1(x^2)+x^{2\iota+2}\varphi_1(a_{1}(x^2))+x\kappa_1(x^2)+x^{2\iota+1}\varphi_1(\kappa_{1}(x^2))\\
	= &(a_1(x)+x^{\iota+1}\varphi_1(a_{1}(x)))^2+x(\kappa_1(x)+x^{\iota}\varphi_1(\kappa_{1}(x)))^2.
	\end{split}
	\end{equation}
	
	From \Cref{prop:typeR} the polynomial $\kappa_1(x)+x^{\iota}\varphi_1(\kappa_{1}(x))$ is either zero or $u(x)$, depending on the parity of $\iota$, and $x=1$ is a root of $u(x)$.
	
	If
	$\gcd(a_1(x^2)+x\kappa_1(x^2)+x^{2\iota+1}(a_1(x^2)+x\varphi_2(\kappa_1(x^2)), x^{4n}-1)\not= x-1$ then it could be that $x=1$ is not a root of $a_1(x^2)+x\kappa_1(x^2)+x^{2\iota+1}(a_1(x^2)+x\varphi_2(\kappa_1(x^2))$ and so neither is a root of $a_1(x)+x^{\iota+1}\varphi_1(a_{1}(x))$. \\Hence, $\gcd(a_1(x)+x^{\iota+1}\varphi_1(a_{1}(x)),x^{2n}-1)\not= x-1$.
	
	Also, it could be that $x-1$ is a root of $a_1(x^2)+x\kappa_1(x^2)+x^{2\iota+1}(a_1(x^2)+x\varphi_2(\kappa_1(x^2))$, but there are more roots in that polynomial, for instance $x=w$, where $w\not=1$ is a root of $x^{4n}-1$. In this case, $w$ is also a root of $\kappa_1(x)+x^{\iota}\varphi_1(\kappa_{1}(x))$ and so, from (\ref{eq:33}), a root of $(a_1(x)+x^{\iota+1}\varphi_1(a_{1}(x)))^2$. Since $x^{4n}-1=(x^{2n}-1)^2$, $w$ is also a root of $x^{2n}-1$ and we obtain $\gcd(a_1(x)+x^{\iota+1}\varphi_1(a_{1}(x)),x^{2n}-1)\not= x-1$. The statement is proven.\qed
\end{proof}

\begin{proposition}\label{prop:double}
	Let  $C=\langle \ba,\bb \rangle$ be an $\HFP$-code of type Q, length $4n$ and $k=2$. Then, there exists two $\HFP$-codes $E$ of type Q, length $8n$, one with dimension of the kernel $k=2$ and another with dimension of the kernel $k=1$. In both cases, if $\rank(C)=2n$ then $\rank(E)=4n$.
\end{proposition}
\begin{proof}
	Let $(a_1(x), a_2(x))$, $(b_1(x), b_2(x))$ and $(\kappa_1(x),\kappa_2(x))$ be the polynomial representation associated to $\ba,\bb,\bkappa=\ba^{\iota}\bb$, respectively, in code $C$.
	We will construct code $E$ taking $\bA,\bKappa=\bA^{2\iota+1}\bB$ as its generators, where the polynomial representations are $(A_1(x), A_2(x))$ and $(K_1(x),K_2(x))$ with
	\begin{equation*}
	\begin{split}
	A_i(x)&=a_i(x^2)+x\kappa_i(x^2)\in \F[x]/x^{4n}-1,\\
	K_i(x)&=1+x^2+\ldots+x^{4n-2}\in \F[x]/x^{4n}-1,
	\end{split}
	\end{equation*}
	respectively.
	We will also take $\pi_\bA=(1,2,\ldots,4n)(4n+1,4n+2,\ldots,8n)$, $\pi_\bB=(1,8n)(2,8n-1)\cdots(4n,4n+1)$ and $\pi_{\bKappa}=\pi_{\bA}^{2\iota+1}\pi_{\bB}$.
	
	It is easy to see that the full propelinear properties of $C$ are maintained in $E$, so $E$ is a full propelinear code.
	
	To show that $E$ is Hadamard, from \Cref{prop:vice}, we need to show that $\wt(\bA^i)=\wt(\bA^i+\bKappa)=4n$, for $i\in\{1,\ldots,4n-1\}$ and $\wt(\bA^{4n})=8n$, $\wt(\bA^{4n}+\bKappa)=4n$.
	Let $A^{(i)}(x)$ be any of the two polynomials associated to the element $\bA^i$, say the first one (for the other polynomial the argument will be the same). Also say $a(x)^{(i)}$ and $\kappa(x)^{(i)}$ the first of the two polynomials associated to $\ba^i,\bkappa^i$, respectively.
	We have
	\begin{equation*}
	\begin{split}
	A^{(2i)}(x)= &(1+x+\ldots+x^{2i-1})a(x^2)+x(1+x+\ldots+x^{2i-1})\kappa(x^2)\\
	=&(1+x^2+\ldots+x^{2i-2})a(x^2)+x(1+x^2+\ldots+x^{2i})\kappa(x^2)\\
	+& x(1+x^2+\ldots+x^{2i-2})a(x^2)+x^2(1+x^2+\ldots+x^{2i})\kappa(x^2)\\
	=& a^{(i)}(x^2)+x\kappa^{(i)}(x^2)+xa^{(i)}(x^2)+x^2\kappa^{(i)}(x^2)
	\end{split}
	\end{equation*}
	We have $x^2\kappa^{(i)}(x^2)=\kappa^{(i)}(x^2)+ \xi_i u(x^2)$, where $\xi_i$ is 0 or 1 depending on $i$ is even or odd, respectively.
	Hence, $A^{(2i)}(x)= p_1(x^2)+xp_2(x^2)$, where $p_1(x^2)=a^{(i)}(x^2)+ \kappa^{(i)}(x^2)+\xi_i u(x^2)$ and $p_2(x^2)=a^{(i)}(x^2)+ \kappa^{(i)}(x^2)$.
	Since $\bkappa, \bu\in K(C)$ we have that both, $p_1(x)$ and $p_2(x)$ are in $C$ and so, $\wt(\bA^i)=4n$ and $\wt(\bA^i+\bKappa)=4n$, for even exponents $i\in\{2,4,\ldots,4n\}$, except for $i=4n$ in which case $\wt(\bA^{4n})=8n$.
	For odd exponents $i\in\{1,3,\ldots,4n-1\}$, using the same decomposition as for the even case, we have
	\begin{equation*}
	\begin{split}
	A^{(2i+1)}(x)=&a^{(i+1)}(x^2)+x a^{(i)}(x^2) + x \kappa^{(i+1)}(x^2)+x^2\kappa^{(i)}(x^2) \\
	=&q_1(x^2)+xq_2(x^2),
	\end{split}
	\end{equation*}
	where
	\begin{equation*}
	\begin{split}
	q_1(x^2)=&a^{(i+1)}(x^2)+ \kappa^{(i)}(x^2)+\xi_i u(x^2)\\
	p_2(x^2)=&a^{(i)}(x^2)+ \kappa^{(i+1)}(x^2).
	\end{split}
	\end{equation*}
	Both, $q_1(x)$ and $q_2(x)$ are in $C$ so, $\wt(\bA^i)=4n$ and $\wt(\bA^i+\bKappa)=4n$.
	
	Regarding the kernel of $E$, we show that $\bKappa \in K(E)$. To do this, we prove that $\pi_\bKappa\in \Aut(E)$. It is clear that $\pi_\bKappa(\bKappa)$ is either $\bKappa$ or $\bKappa+\bu$. In any case $\pi_\bKappa(\bKappa)\in E$.
	Also, $\pi_\bKappa(\bA^i)=\bKappa + \bA^i+ \pi_{\bA^i}(\bKappa)$. We have that $\bKappa +\pi_{\bA^i}(\bKappa)$ is either $\be$ or $\bu$ so, in any case, $\pi_\bKappa(\bA^i) \in E$. Hence, the dimension of the kernel is $k\geq 2$ and from \Cref{rk} we have $k=2$.
	
	For the rank, the result is clear from \Cref{prop:double1} and \Cref{prop:double2}.
	
	Finally, note that the constructed code $E$ has dimension of the kernel equal to 2. Now, from \Cref{prop:trans}, we can construct the transposed code of $E$, which has dimension of the kernel equal to 1.
	The statement is proven.\qed
\end{proof}

The construction in the above \Cref{prop:double} is more specific than a previous construction in \cite[Prop. 1]{itoIII}. There, it is shown that given an Hadamard group of type Q we can obtain an Hadamard group of type Q with double length. Here, in \Cref{prop:double} we show that from an $\HFP$-code of type Q with dimension of the kernel $k=2$ and maximum rank, we obtain an $\HFP$-code of type Q with double length, dimension of the kernel $k=2$ and maximum rank. Moreover, if we are looking for $\HFP$-codes of double length and dimension of the kernel $k=1$, we can obtain them applying \Cref{prop:trans}.

\vspace*{0,3cm}
\textbf{Acknowledgment.}
This work has been partially supported by the Spanish MICINN grants
	TIN2016-77918-P, MTM2015-69138-REDT and the Catalan AGAUR grant
	2014SGR-691.

\end{document}